\documentclass[reqno, 12pt]{amsart}

\pdfoutput=1

\usepackage{latexsym}
\usepackage[centertags]{amsmath}
\usepackage{amsfonts}
\usepackage{amssymb}
\usepackage{amsthm}
\usepackage{newlfont}
\usepackage{graphics}
\usepackage{color}

\newtheorem{thm}{Theorem}

\newtheorem{lem}[thm]{Lemma}
\newtheorem{prop}[thm]{Proposition}

\theoremstyle{mydefinition}

\newtheorem{rema}[thm]{Remark}

\theoremstyle{myremark}

\allowdisplaybreaks[1]

\newcommand\MATHCAL[1]{\mathcal{#1}}
\def\pa[1]{\frac{\partial}{\partial x}}

\def\diag{\mathrm{diag}}

\newcommand{\qfac}[1]{(q)_n}

\title[Shifted Periodic Continued Fractions]{Hankel Determinants and Shifted Periodic Continued Fractions}

\author{Ying Wang$^1$, Guoce Xin$^{2,*}$ and Meimei Zhai$^3$}

 \address{ $^{1,2}$School of Mathematical Sciences, Capital Normal University,
 Beijing 100048, PR China\\
 $^3$Rizhao No.1 Middle School of Shandong, Shandong Province 276800, PR China}
 \email{$^1$\texttt{wangying.cnu@gmail.com}\ \  \& $^2$\texttt{guoce.xin@gmail.com}\ \& $^3$\texttt{zhaimeimei3@163.com}}
\date{June 26, 2018}

\begin{document}

\maketitle

\begin{abstract}
Sulanke and Xin developed a continued fraction method that applies to evaluate Hankel determinants corresponding to quadratic generating functions.
We use their method to give short proofs of Cigler's Hankel determinant conjectures, which were proved recently by Chang-Hu-Zhang using direct determinant computation.
We find that shifted periodic continued fractions arise in our computation. We also discover and prove some new nice Hankel determinants relating to lattice paths
with step set $\{(1,1),(q,0), (\ell-1,-1)\}$ for integer parameters $m,q,\ell$. Again shifted periodic continued fractions appear.
\end{abstract}

\noindent
\begin{small}
 \emph{Mathematic subject classification}: Primary 15A15; Secondary 05A15, 11B83.
\end{small}

\noindent
\begin{small}
\end{small}

\section{Introduction}
%
%
%
%
%

Let $A=(a_0,a_1,a_2\cdots)$ be a sequence, and denote by $A(x)=\sum_{n\geq0}a_nx^n$ its generating function. Define the shifted Hankel matrices (or determinants) of $A$ or $A(x)$  by
 $$\mathcal{H}_n^{(k)}(A)= \mathcal{H}_n^{(k)}(A(x)) =(a_{i+j+k})_{0\leq i,j\leq n-1},\qquad \text{ and } H_n^{(k)}(A) = \det \mathcal{H}_n^{(k)}(A).$$
 We shall write $H_n(A)$ for $H_n^{(0)}(A)$ and $H_n^1(A)$ for $H_n^{(1)}(A)$.
In convention we set $H_0^{(k)}=1$. If it is clear from the context, we will omit the $A$ or $A(x)$.
For instance, the sequence of Catalan number: $1,\ 1,\  2,\  5,\  14,\  42,\  132,\  429,\  1430,\  4862\ \cdots$ yields
\begin{equation*}
\MATHCAL{H}_1=\left[1\right],
\MATHCAL{H}_2={
\left[ \begin{array}{cc}
1 & 1 \\
1& 2
\end{array}
\right ]},
\MATHCAL{H}_3=\left[ \begin{array}{ccc}
1 & 1 & 2\\
1 & 2 & 5\\
2 & 5 & 14
\end{array}
\right ],
\MATHCAL{H}_4=\left[ \begin{array}{cccc}
1 & 1 & 2 & 5\\
1 & 2 & 5 & 14\\
2 & 5 &14 & 42\\
5 & 14&42 &132
\end{array}
\right ].
\end{equation*}

In recent years, a considerable amount of work has been devoted to Hankel determinants of path counting numbers, especially for weighted counting of lattice paths with up step $(1,1)$, level step $(\ell,0), \ \ell \ge 1$, and down step $(m-1,-1), \ m\ge 2$. Many of such Hankel determinants have attractive compact closed formulas, such as that of Catalan numbers \cite{J.M.E.Mays J.Wojciechowski}, Motzkin numbers \cite{M. Aigner,J.Cigler}, and Schr\"{o}der numbers \cite{R. A Brualdi and S. Kirkland.}. For instance, Motzkin numbers count lattice paths from $(0,0)$ to $(n,0)$ with step set $\{(1,1), (1,0), (1,-1)\}$ that never go below the horizontal axis. Partial Motzkin paths (similar to Motzkin paths but from $(a,0)$ to $(n,b)$) were considered in \cite{J.Cigler-C.Krattenthaler} and \cite{C.Krattenthaler-D.Yaqubi}, where many nice determinant formulas were discovered.

Many methods have been developed for evaluating Hankel determinants using their corresponding generating functions. One of the basic tools is
the method of continued fractions, either by J-fractions in Krattenthaler \cite{C. Krattenthaler.} or Wall \cite{H. S. Wall} or by S-fractions in
Jones and Thron \cite[Theorem 7.2]{W. B. Jones and W. J. Thron}. However, both of these methods need the condition that the determinant
can never be zero, a condition not always present in our study.

Our point of departure is that such lattice paths have quadratic generating functions, so that Sulanke-Xin's continued fraction method applies to evaluate
their Hankel determinants. In \cite{R.Sulanke and Xin} Sulanke and Xin used Gessel-Xin's continued fraction method \cite{Gessel and Xin} to evaluate the Hankel determinants for lattice paths with step set $\{(1,1), (3,0), (1,-1)\}$. 
\begin{prop}\cite{R.Sulanke and Xin}
Let $F(x)$ be determined by $ F(x)=1+x^3F(x)+x^2F(x)^2 $. Then
\[(H_n)_{n\geq1}^{14}=(1,1,0,0,-1,-1,-1,-1,-1,0,0,1,1,1).\] Moreover, if $m, n\geq0 $ with $(n-m)\equiv_{14}0,$ then $H_m=H_n.$
\end{prop}
\noindent
They indeed
defined a quadratic transformation $\tau$ (see Proposition \ref{xinu(0,i,ii)}) such that
there is a simple relation between the Hankel determinants of $F$ and $\tau(F)$. Now if we let
$F_0(x)=F(x)$ and apply the transformation $\tau$, then we obtain the following periodic continued fractions:
\begin{align} \label{e-periodic-F}
  F_0(x) \mathop{\longrightarrow}\limits^\tau F_1(x) \mathop{\longrightarrow}\limits^\tau \cdots \mathop{\longrightarrow}\limits^\tau F_5(x)=F_0(x).
\end{align}
Thus one easily deduce that $H_n(F)=-H_{n-7}(F).$ This method is now called Sulanke-Xin's continued fraction method,
and has been applied systematically in \cite{R.Sulanke and Xin}. It was later applied to solve Somos recurrence problem in \cite{X. K. Chang-X. B. Hu and G. Xin,Xin Somos4}.

Our basic tools are Gessel-Xin's and Sulanke-Xin's continued fraction methods.
These methods work nicely for quadratic generating functions, and they allow the Hankel determinants to be zero.
We find it natural to define a \emph{shifted periodic continued fractions} to be of the following form:
$$F_0^{(p)}(x) \mathop{\longrightarrow}\limits^\tau F_1^{(p)}(x)\mathop{\longrightarrow}\limits^\tau \cdots \mathop{\longrightarrow}\limits^\tau F_q^{(p)}(x)=F_0^{(p+1)}(x), $$
where $p$ is an additional integer parameter. It reduces to periodic continued fractions if $F_i^{(p)}$ is independent of $p$.
Once the pattern is guessed, the proof is automatic and the Hankel determinants are easy to evaluate. For instance, our first
example of periodic continued fractions is \eqref{e-periodic-F}. An example of shifted periodic continued fractions arise naturally from the evaluation
of shifted Hankel determinants of Catalan numbers. See Section \ref{sec:catalan}.

One purpose of this paper is to give short proofs of Cigler's conjectures \cite{J.Cigler} on the Hankel determinants of
the three sequences $\{c(n,m,a,b)\}_{n=0}^\infty,\{C(n,m,a,b,t)\}_{n=0}^\infty,$ $\{g(n,m,a,b)\}_{n=0}^\infty,$ whose generating functions
respectively satisfy the functional equations
 \begin{align}
 f_m(x,a,b) &=\sum_{n\geq 0} c(n,m,a,b) x^n  \nonumber\\
 &= 1+a x  f_m(x,a,b)+b x^m  f_m(x,a,b)^2,\label{generating function f} \\
 F_m(x,a,b,t) &=\sum_{n\geq 0} C(n,m,a,b,t) x^n \nonumber\\
 &= 1+(a+t) x  F_m(x,a,b,t)+b x^m F_m(x,a,b,t) f_m(x,a,b),\label{generating function F} \\
 G_m(x,a,b) &=\sum_{n\geq 0} g(n,m,a,b) x^n \nonumber\\
 &= 1+a x  G_m(x,a,b)+2 b x^m G_m(x,a,b) f_m(x,a,b),\label{generating function G}
 \end{align}
where $a, b, t$ are arbitrary complex numbers and $m$ is a fixed positive integer.
Cigler  first considered the Hankel determinants $H_n^{(r)}(f_m(x)),H_n^{(r)}(F_m(x)),H_n^{(r)}(G_m(x))$
for $r=0,1,2.$ He used orthogonal polynomials method \cite{J.Cigler, C. Krattenthaler.,X.G.},
Gessel-Viennot-Lindstr\"{o}m theorem \cite{D.M. Bressoud and A.M. Americaof.,I. Gessel and G. Viennot.,R.Sulanke and Xin}
and the continued fraction method \cite{R.Sulanke and Xin,Gessel and Xin} to compute the Hankel determinants and successfully obtained some results.
He also listed several conjectures, which have been proved by Chang-Hu-Zhang \cite{X. K. Chang. X. B. Hu and Y. N. Zhang} using direct determinants computation. Our proofs are
complete and short.

We remark that Krattenthaler described several methods to evaluate
determinants and listed many known determinant evaluations in \cite{C. Krattenthaler.,C. Krattenthaler.1999}.
Usually if we know the Hankel determinants $H_n(A(x))$ and $H_n^1(A(x))$, then $H_n^{(k)}(A(x))$ (if never vanish) can be recursively computed using the condensation formula for determinants \cite{C. Krattenthaler.,R.A. Brualdi and H. Schneider.,C.L. Dodgson.}:
\begin{align}
  \label{e-condensation}
  H_n^{(k)}(A(x))H_{n-2}^{(k+2)}(A(x))=H_{n-1}^{(k)}(A(x))H_{n-1}^{(k+2)}(A(x))-(H_{n-1}^{(k+1)}(A(x)))^2.
\end{align}

The paper is organized as follows. Section \ref{sec:Gessel-Xin-method} introduces Gessel-Xin's continued fraction method. We also derive some lemmas for proving Cigler's ex-conjectures.
Section \ref{sec:sulanke-xin-trans} introduces Sulanke-Xin's continued fraction method, especially the quadratic transformation $\tau$. As an application, we evaluate shifted Hankel determinants for Catalan numbers.
We will see that shifted periodic continued fractions arise naturally. In Section \ref{sec:short proof 2}, we complete the proof of Cigler's ex-conjectures, where shifted periodic continued fractions also appear.
Section \ref{sec:periodic} includes two examples of artificial lattice paths, whose generating functions have periodic continued fractions. Section \ref{sec:four-classes} considered
a special family of generating functions that are related to lattice paths with step set $\{(1,1),(q,0), (\ell-1,-1)\}$ for integer parameters $m,q,\ell$. We find nice Hankel determinants for four classes of  parameters $m,q,\ell$. Again
shifted periodic continued fractions appear.

\section{Gessel-Xin's continued fraction method\label{sec:Gessel-Xin-method}}
In this section, we illustrate Gessel-Xin's continued fraction method. This method is based on the evaluation of determinants by using generating functions.

\subsection{Basic rules}
For an arbitrary two variable generating function $D(x,y)=\sum_{i,j=0}^{\infty}d_{i,j}x^iy^j$, let $[D(x,y)]_n$ be the determinant of the $\ n\times n$ matrix
$(d_{i,j})_{0\leq{i,j}\leq{n-1}}.$ It may be intuitive to see that ordinary and shifted Hankel determinants, $H_n(A)$ and $H^1_n(A)$, can be expressed as
\[H_n(A)=\Big[\frac{xA(x)-yA(y)}{x-y}\Big]_n,\qquad H^1_n(A)=\Big[\frac{A(x)-A(y)}{x-y}\Big]_n.\]

There are three simple rules to transform the determinant $[D(x, y)]_n$ to another determinant. See \cite{Gessel and Xin} for applications.

\noindent
\emph{Constant Rules.} Let $c$ be a non-zero constant. Then
$$[cD(x, y)]_n=c^n[D(x, y)]_n,\quad  \text{ and} \quad [D(cx, y)]_n = c^{\binom{n}{2}}[D(x, y)]_n=[D(x, cy)]_n.$$
\emph{Product Rules.} If $u(x)$ is any formal power series with $u(0) = 1$, then
$$[u(x)D(x, y)]_n = [D(x, y)]_n=[u(y)D(x, y)]_n.$$
\emph{Composition Rules.} If $v(x)$ is any formal power series with $v(0) = 0$ and $v'(0) = 1$, then
$$[D(v(x), y)]_n = [D(x, y)]_n=[D(x, v(y))]_n.$$

The constant rules are clear. The product and composition rules hold because the transformed determinants are obtained
from the original one by a sequence of elementary row or column operations. 

In our applications, only the constant rules and multiplication  rules will be used.
In addition, we use the following simple fact of determinants:
\begin{align}
 [c + D(x,y)]_n = c \det (d_{i,j} )_{1\le i,j \le n-1} + [D(x,y)]_n. \label{e-cDxy}
\end{align}

\subsection{Simple derivation of Cigler's ex-conjectures: part I}
We will give simple evaluation of the Hankel determinants of $f_m, F_m$ and $G_m$, as defined in \eqref{generating function f},
\eqref{generating function F} and \eqref{generating function G}, respectively. The computation is divided into two parts.
This subsection is the first part, which includes some lemmas proved by the product rules.
Section \ref{sec:short proof 2} is the second part, which includes the detailed proofs. Some of the results need to use shifted periodic continued fractions.

We will use the fractional representation for the defining functional equation.
For instance, \eqref{generating function f} is rewritten in the following form
\begin{align}
 f_m(x,a,b) &= \frac{1}{1-ax-b x^m f_m(x,a,b)}. \label{e-cf-fm}
\end{align}

We first evaluate the following Hankel determinants $H_n(x^jf_m)$.
\begin{lem}\label{f_m(x))}
For $m\geq 2$ and $0\leq j\leq m-2$, we have
\begin{align*}
H_{mn}^{-j}(f_m)&=H_{mn}(x^{j}f_m)=(-1)^{\binom{j+1}{2}n}(-1)^{\binom{m-j-1}{2}n}b^{mn^2-(j+1)n}\\
H_{mn+j+1}^{-j}(f_m)&=H_{mn+j+1}(x^{j}f_m)=(-1)^{\binom{j+1}{2}(n+1)}(-1)^{\binom{m-j-1}{2}n}b^{mn^2+(j+1)n}
\end{align*}
and $H_n^{-j}(f_m)=H_{n}(x^{j}f_m)=0$ for all other $n$.

For $j=m-1$ we have
\begin{align}
  H_{n}(x^{m-1} f_m)= (-1)^{\binom{m}{2}} b^{n-m} H^1_{n-m}(f_m). \label{e-fm-j=m-1}
\end{align}
\end{lem}
\begin{proof}By definition $H_{n}^{-j}(f_m)=H_{n}(x^{j}f_m)$ holds true for all $j\ge 0$.

\begin{align*}
H_{n}&  (x^{j}f_m(x))=\left[\frac{x^{j+1}f_m(x)-y^{j+1}f_m(y)}{x-y}\right]_n   \qquad (\text{apply }  \eqref{e-cf-fm} )  \\
&=\left[\frac{\frac {x^{j+1}} {1-ax-b x^m f_m(x)}-\frac {y^{j+1}} {1-ay-b y^m f_m(y)}}{(x-y)}\right]_n    \qquad (\times (1-ax-b x^m f_m(x))(1-ay-b y^m f_m(y)))   \\
&=\left[\frac{x^{j+1} (1-ay -by^m f_m(y))-y^{j+1} (1- ax -b x^m f_m(x)) }{(x-y)}\right]_n \\
&=\left[\frac{(x^{j+1}-y^{j+1})-axy (x^{j}-y^{j})}{x-y}+bx^{j+1}y^{j+1}\frac{ x^{m-j-1} f_m(x)- y^{m-j-1} f_m(y) }{(x-y)}\right]_n.
\end{align*}
Now the resulting matrix is blocked diagonal of the form $\diag(A,B)=\begin{pmatrix}
                                                                       A & 0 \\
                                                                       0 & B \\
                                                                     \end{pmatrix}
$, where
$$A= \begin{pmatrix}
                                                                   0 & \cdots & 0 & 1 \\
                                                                   0 & \ddots & 1 & -a \\
                                                                   \vdots & \ddots & \cdots & 0 \\
                                                                   1 & -a & \cdots  & 0 \\
                                                                 \end{pmatrix}_{\!\!\!\!(j+1)\times (j+1)}, \qquad \qquad B= b\cdot \MATHCAL{H}_{n-j-1}^1(x^{m-j-1}f_m(x)).$$
Note that we need the condition $0\le j \le m-1$ here. Also note that when $j<m-1$ we have $H_{n-j-1}^1\Big(x^{m-j-1}f_m(x)\Big)= H_{n-j-1}\Big(x^{m-j-2}f_m(x)\Big).$

                                                                  Thus the case for $j=m-1$ follows. For $0\le j\le m-2$,
                                                                   we obtain
\begin{align}
  H_n(x^{j}f_m)= (-1)^{\binom{j+1}{2}}b^{n-j-1}H_{n-j-1}^0(x^{m-j-2}f_m).\label{e-A0A1}
\end{align}

Similar calculation gives
\begin{align*}
H_{n}&(x^{m-j-2}f_m(x))=\left[\frac{x^{m-j-1}f_m(x)-y^{m-j-1}f_m(y)}{(x-y)}\right]_n=\left[\frac{\frac{x^{m-j-1}}{(1- ax -bx^m f_m(x))}-\frac{y^{m-j-1}}{(1-ay -by^m f_m(y))}}{(x-y)}\right]_n\\
&=\left[\frac{x^{m-j-1} (1-ay -by^m f_m(y))-y^{m-j-1} (1- ax -bx^m f_m(x))}{(x-y)}\right]_n\\
&=\left[\frac{x^{m-j-1}-y^{m-j-1}-ax y(x^{m-j-2}-y^{m-j-2})}{x-y}+bx^{m-j-1}y^{m-j-1}\frac{ x^{j+1}f_m(x)-y^{j+1}f_m(y)}{(x-y)}\right]_n \\
&=(-1)^{\binom{m-j-1}{2}}b^{n-m+j+1}H_{n-m+j+1}(x^{j}f_m(x)).
\end{align*}
Then together with \eqref{e-A0A1} we have
\begin{gather}
  H_{n}(x^{j}f_m(x))=(-1)^{\binom{j+1}{2}}(-1)^{\binom{m-j-1}{2}}b^{2n-m-j-1}H_{n-m}(x^{j}f_m(x)).\label{xjfmf}
\end{gather}
This gives a recursion and to complete the proof we only need the easily verified initial conditions:
i) $H_0(x^{j}f_m(x))=1$; ii) $H_{j+1}(x^{j}f_m(x))=(-1)^{\binom{j+1}{2}}$;  iii) $H_n(x^{j}f_m(x))=0$ for $0\leq n\leq j$ by \eqref{xjfmf};
 iv) $H_n(x^{j}f_m(x))=0$ for $j+2\leq n\leq m-1$.
\end{proof}

More generally, we can use the same method to evaluate the Hankel determinants of $q_m(x,t)$, where
\begin{align}
  q_m(x,t)=\frac{1}{1-\alpha(t) x-\beta x^mf_m(x)}, \label{e-qm}
\end{align}
in which $\alpha(t) $ is a  polynomial in $t$ and $\beta $ is a constant.
Note that when $\alpha(t) =a, \beta =b$, $q_m(x,t)$ specializes to $f_m(x,a,b)$; when $\alpha(t) =a+t, \beta =b$, $q_m(x,t)$ specializes to $F_m(x,a,b)$; when $\alpha(t) =a, \beta =2b$, $q_m(x,t)$ specializes to $G_m(x,a,b)$.
\begin{lem}\label{q(x,t)}
For $m\geq 2$, $1\leq i\leq m-1$, we have
$$H_{n}(x^{i-1}q_m(x,t))=(-1)^{\binom{i}{2}}(-1)^{\binom{m-i}{2}}\beta ^{n-i}b^{n-m}H_{n-m}(x^{i-1}f_m(x)).$$
\end{lem}
\begin{proof} For fixed $i$ with $1\leq i\leq m-1$, by \eqref{e-qm}, we have
\begin{align*}
H_{n}(x^{i-1}q_m(x,t))&=\left[\frac{x^{i}q_m(x,t)-y^{i}q_m(y,t)}{x-y}\right]_n=\left[\frac{\frac {x^{i}} {(1- \alpha(t) x -\beta  x^m f_m(x))}-\frac {y^{i}} {(1-\alpha(t) y -\beta y^m f_m(y))}}{(x-y)}\right]_n\\
&=\left[\frac{x^{i} (1-\alpha(t) y -\beta y^m f_m(y))-y^{i} (1- \alpha(t) x -\beta  x^m f_m(x)) }{(x-y)}\right]_n\\
&=\left[\frac{(x^{i}-y^{i})-\alpha(t) xy({x^{i-1}-y^{i-1}})}{x-y}+\beta x^{i}y^{i}\frac{ x^{m-i} f_m(x)- y^{m-i} f_m(y) }{(x-y)}\right]_n \\
&=(-1)^{\binom{i}{2}}\beta ^{n-i}H_{n-i}(x^{m-i-1}f_m(x)).
\end{align*}
Note that we need the condition $1\leq i\leq m-1$.

By \eqref{e-A0A1} we have
\begin{align*}
  H_{n}(x^{i-1}q_m(x,t))&=(-1)^{\binom{i}{2}}\beta ^{n-i}H_{n-i}(x^{m-i-1}f_m(x)) \\
  & =(-1)^{\binom{i}{2}}(-1)^{\binom{m-i}{2}}\beta ^{n-i}b^{n-m}H_{n-m}(x^{i-1}f_m(x)).
\end{align*}
This completes the proof.
\end{proof}

In order to calculate $H_n^1(f_m)$, $H_n^1(F_m)$ and $H_n^1(G_m)$, we need 
 the following results.
\begin{lem}\label{q(x,t)1}
\begin{align}
H_{n}^1(q_1(x,t))&=\alpha(t) \beta ^{n-1}H_{n-1}^{(2)}(f_1(x))+\beta ^{n}H_{n}(f_1(x)).\label{f-e-q1}\\
H_{n}^1(q_2(x,t))&=\alpha(t) \beta ^{n-1}H_{n-1}^1(f_2(x))-\beta ^{n}b^{n-2}H_{n-2}^1(f_2(x)).\label{f-e-q2}
\end{align}
For $m\geq 3$, we have $$H_{n}^1(q_m(x,t))=\alpha(t) \beta ^{n-1}H_{n-1}^0(x^{m-3}f_m(x))+(-1)^{\binom{m}{2}}\beta ^{n}b^{n-m}H_{n-m}^1(f_m(x)).$$
\end{lem}
\begin{proof} By \eqref{e-qm}, we have
\begin{align*}
H_{n}^1(q_m(x))&=\left[\frac{q_m(x,t)-q_m(y,t)}{x-y}\right]_n=\left[\frac{\frac {1} {(1- \alpha(t) x -\beta  x^m f_m(x))}-\frac {1} {(1-\alpha(t) y -\beta y^m f_m(y))}}{(x-y)}\right]_n\\
&=\left[\frac{(1-\alpha(t) y -\beta y^m f_m(y))-(1- \alpha(t) x -\beta  x^m f_m(x)) }{(x-y)}\right]_n\\
&=\left[\alpha(t) +\beta \frac{ x^{m} f_m(x)- y^{m} f_m(y) }{(x-y)}\right]_n.
\end{align*}
Note that
\item[i)] By \eqref{e-cDxy}, when $m=1$, we have
$$H_{n}^1(q_1(x,t))=\alpha(t) \beta ^{n-1}H_{n-1}^{(2)}(f_1(x))+\beta ^{n}H_{n}(f_1(x)).$$
\item[ii)] By \eqref{e-cDxy}, when $m=2$, we have \begin{align*}
H_{n}^1(q_2(x,t))&=\alpha(t) \beta ^{n-1}H_{n-1}^1(f_2(x))+\beta ^{n}H_{n}(xf_2(x))\\
&=\alpha(t) \beta ^{n-1}H_{n-1}^1(f_2(x))-\beta ^{n}b^{n-2}H_{n-2}^1(f_2(x)).\quad (\text{by } \eqref{e-fm-j=m-1})
\end{align*}
\item[iii)] By \eqref{e-cDxy}, when $m \geq 3$, we have
\begin{align*}
H_{n}^1(q_m(x))&=\alpha(t) \beta ^{n-1}H_{n-1}(x^{m-3}f_m(x))+\beta ^{n}H_{n}(x^{m-1}f_m(x))\\
 &=\alpha(t) \beta ^{n-1}H_{n-1}(x^{m-3}f_m(x))+(-1)^{\binom{m}{2}}\beta ^{n}b^{n-m}H_{n-m}^1(f_m(x)).\quad (\text{by  } \eqref{e-fm-j=m-1})
\end{align*}
Then we complete the proof.
\end{proof}
\begin{lem}\label{q_1(x,t)}
We have
\begin{align*}
H_{n}(q_1(x,t))=a\beta ^{n-1}b^{n-2}H_{n-2}^{(2)}(f_1(x))+(\beta b)^{n-1}H_{n-1}(f_1(x)).
\end{align*}
\end{lem}
\begin{proof}By \eqref{e-qm}, we have
\begin{align*}
H_{n}(q_1(x,t))&=\left[\frac{xq_1(x,t)-yq_1(y,t)}{x-y}\right]_n=\left[\frac{\frac {x}{1- \alpha(t) x -\beta  x f_1(x)}-\frac {y}{1-\alpha(t) y -\beta yf_1(y)}}{x-y}\right]_n\\
&=\left[\frac{x(1-\alpha(t) y -\beta yf_1(y))-y(1- \alpha(t) x -\beta  x f_1(x))}{(x-y)}\right]_n\\
&=\left[1+\beta xy\frac{f_1(x)- f_1(y) }{(x-y)}\right]_n \\
&=\beta ^{n-1}H_{n-1}^1(f_1(x)).
\end{align*}
Now  by \eqref{f-e-q1} with respect to $f_1$, we obtain
\begin{gather*}
 H_{n}^1(f_1(x))=ab^{n-1}H_{n-1}^{(2)}(f_1(x))+b^nH_{n}(f_1(x)).
\end{gather*}

Therefore we have
\begin{align*}
H_{n}(q_1(x,t))&=\beta ^{n-1}H_{n-1}^1(f_1(x))=a\beta ^{n-1}b^{n-2}H_{n-2}^{(2)}(f_1(x))+(\beta b)^{n-1}H_{n-1}(f_1(x)).
\end{align*}
\end{proof}

\section{Sulanke and Xin's quadratic transformation\label{sec:sulanke-xin-trans}}

In this section, we introduce the continued fraction method of Sulanke and Xin \cite{R.Sulanke and Xin}.
This is the main tool of this paper.

\subsection{The transformation $\tau$}
Suppose the generating function $F(x)$ is the unique solution of a quadratic functional equation which can be written as
\begin{gather}
  F(x)=\frac{x^d}{u(x)+x^kv(x)F(x)},\label{xinF(x)}
\end{gather}
where $u(x)$ and $v(x)$ are rational power series with nonzero constants, $d$ is a nonnegative integer, and $k$ is a positive integer.
We need the unique decomposition of $u(x)$ with respect to $d$: $u(x)=u_L(x)+x^{d+2}u_H(x)$ where $u_L(x)$ is a polynomial of degree at most $d+1$ and $u_H(x)$ is a power series.
Then Propositions 4.1 and 4.2 of \cite{R.Sulanke and Xin} can be summarized as follows.

\begin{prop}\label{xinu(0,i,ii)}
Let $F(x)$ be determined by  \eqref{xinF(x)}. Then the quadratic transformation $\tau(F)$ of $F$ defined as follows gives close connections
between   $H(F)$ and $H(\tau(F))$.
\begin{enumerate}
\item[i)] If $u(0)\neq1$, then $\tau(F)=G=u(0)F$ is determined by $G(x)=\frac{x^d}{u(0)^{-1}u(x)+x^ku(0)^{-2}v(x)G(x)}$, and $H_n(\tau(F))=u(0)^{n}H_n(F(x))$;

\item[ii)] If $u(0)=1$ and $k=1$, then $\tau(F)=x^{-1}(G(x)-G(0))$, where $G(x)$ is determined by
$$G(x)=\frac{-v(x)-xu_L(x)u_H(x)}{u_L(x)-x^{d+2}u_H(x)-x^{d+1}G(x)},$$
and we have
$$H_{n-d-1}(\tau(F))=(-1)^{\binom{d+1}{2}}H_n(F(x));$$

\item[iii)] If $u(0)=1$ and $k\geq2$, then $\tau(F)=G$, where $G(x)$ is determined by
$$G(x)=\frac{-x^{k-2}v(x)-u_L(x)u_H(x)}{u_L(x)-x^{d+2}u_H(x)-x^{d+2}G(x)},$$
and we have
$$H_{n-d-1}(\tau(F))=(-1)^{\binom{d+1}{2}}H_n(F(x)).$$\label{xinu(0)(ii)}

\end{enumerate}

\end{prop}

\begin{rema}
  The proposition is easily implemented by Maple, so most of our examples only exhibit the important steps. For an example with full details, see the proof of
  Theorem \ref{t-for-Catalan}, where we have parameter exponents and then have to compute by hand.
\end{rema}

\subsection{On shifted Hankel determinants of Catalan numbers\label{sec:catalan}}
The well-known Catalan generating function $C(x)=\sum_{n\ge 0} C_n x^n$, where $C_n=\frac{1}{n+1}\binom{2n}{n}$, satisfies the quadratic functional equation
$C(x)=1+xC(x)^2$. Its equivalent form
\begin{align*}
  C(x)=\frac{1}{1-xC(x)}
\end{align*}
can be used to produce the so-called  S-continued fractions:
$$C(x)=\frac{1}{1-\displaystyle\frac{x}{1-\displaystyle\frac{x}{1-\cdots}}}.$$

Using this representation, one can show that $H_n(C(x))=1,\ \ H_n^1(C(x))=1$. Then using induction and the condensation formula \eqref{e-condensation}, one can show (see, e.g.,\cite{J.M.E.Mays J.Wojciechowski})  that
$$H_n^{(2)}(C(x))=n+1,\ \ H_n^{(i)}(C(x))=\prod_{j=1}^{i-1}\prod_{k=1}^{j}\frac{2n+j+k}{j+k}.$$

Our attempt in this subsection is to use shifted periodic continued fractions to prove the above results.
However, the computation becomes complicated for $H_n^{(i)}(C(x))$ when $i\ge 3$.
\begin{proof}[Sketched computation of $H_n^{(i)}(C(x))$ for $i\le 2$]

\noindent
Case $i=0,1$:
Applying Proposition \ref{xinu(0,i,ii)} gives $H_n(C(x)) =H_{n-1}^1(C(x))= H_{n-1}(C_1(x)),$ where
$$C_1(x)=\frac{C(x)-1}{x}=\frac{1}{ 1-2\,x-{x}^{2}C_1(x)}.$$

Applying Proposition \ref{xinu(0,i,ii)} again  gives $H_{n-1}(C_1(x))=H_{n-2}(C_1(x))$. Thus we obtain
$H_n(C(x))=H_{n-1}^1(C(x))=H_{n-1}(C_1(x))=\cdots =H_0(C_1(x))=1$.

\noindent
Case $i=2$: First we write $H_n^{(2)}(C)=H_n(A_0)$ where
$$A_0(x)=\frac{C(x)-x-1}{x^2}={\frac {x+2}{-{x}^{3}A_0(x)-2\,{x}^{2}-2\,x+1}}.$$

Applying Proposition \ref{xinu(0,i,ii)} gives  $H_n(A_0)=2^n H_{n-1}(A_1)$, where $$A_1(x)=-\frac {x+6}{4({x}^{2}A_1(x) \left( x+2 \right) +{x}^{2}+4\,x-2)}.$$
Repeat application of Proposition \ref{xinu(0,i,ii)} suggests that we shall define
$$A_p(x)=-\frac{p^2\left(x+(p+1)(p+2)\right)}{(p+1)^2x^2(x+(p+1)p)A_p(x)+2p(p+1)x^2+2p(p+1)^3x-p(p+1)^3}, \qquad p\geq 1$$
and obtain
\begin{align*}
  H_n(A_{p})&=\left( {\frac {(p+1)^2-1}{ (p+1)^2}}\right)^nH_{n-1}(A_{p+1}), \qquad p\geq1.
\end{align*}
It then follows that
 $$ H_n^{(2)}(C(x))=2^n\cdot\left(\frac{2^2-1}{2^2}\right)^{n-1}\!\!\!\!\cdot\left(\frac{3^2-1}{3^2}\right)^{n-2}\!\!\!\!\cdot \cdots \cdot\left(\frac{n^2-1}{n^2}\right)=n+1.$$
%
This completes the proof.
\end{proof}

\begin{rema}
  The continued fractions for $C_1(x)$ can be written as $C_1(x)\mathop{\longrightarrow}\limits^\tau C_1(x)$. This is the simplest periodic continued fractions. One of the simplest shifted periodic continued fractions is
   $A_p(x)\mathop{\longrightarrow}\limits^\tau A_{p+1}(x)$.  Note that
  $C(x) \to C(x)$ is not the periodic continued fractions in our sense.
\end{rema}

\section{Short derivation of Cigler's ex-conjectures: part II \label{sec:short proof 2}}
In this section we complete the computation of $H_n^{(i)}(f_m(x))$, $H_n^{(i)}(F_m(x))$, $H_n^{(i)}(G_m(x))$ for $i=0,1, 2$.
Note that \eqref{conG1}, \eqref{conG2} and most part of Theorem \ref{f-2} were conjectured in \cite{J.Cigler}, and proved in \cite{X. K. Chang. X. B. Hu and Y. N. Zhang}.
Similar idea can be used to compute $H_n^{(3)}(f_m(x))$, $H_n^{(3)}(F_m(x))$, $H_n^{(3)}(G_m(x))$.

\subsection{For $H_n(q_m(x,t))$}
The $m=1$ case has to be considered separately.
\begin{thm}\cite{X. K. Chang. X. B. Hu and Y. N. Zhang,J.Cigler}\label{Hq1}
\begin{align*}
  H_n(f_1(x))&=b^{\binom{n}{2}}(a+b)^{\binom{n}{2}}, \\
  H_n(F_1(x))&=b^{\binom{n}{2}}(a+b)^{\binom{n}{2}},\\
  H_n(G_1(x))&=2^{n-1}b^{\binom{n}{2}}(a+b)^{\binom{n}{2}}.
\end{align*}
\end{thm}
\begin{proof}
Consider the continued fractions
$$f_1(x)=\frac{1}{1-ax-bxf_1(x)}.$$
Repeat application of Proposition \ref{xinu(0,i,ii)} gives
\begin{gather}\label{f1-f1}
 H_n(f_1)=H_{n-1}^1(bf_1)=b^{n-1}H_{n-1}^1(f_1).
\end{gather}
Now let $A_1=\frac{f_1-1}{x}$, and repeatedly apply Proposition \ref{xinu(0,i,ii)} again. We have
 \begin{align*}
H_{n-1}^1(f_1)&=H_{n-1}(A_1), \quad & A_1(x)=\frac{(a+b)}{1-(a+2b)x-bx^2A_1(x)}.\\
  H_{n-1}(A_1)&=(a+b)^{n-1}H_{n-2}(A_2),&A_2(x)={\frac {b\left(a+b\right)}{1-(a+2b)x-{x}^{2}A_2(x)}}.\\
 H_{n-2}(A_2)&=(b(a+b))^{n-2}H_{n-3}(A_3),&A_3(x) = {\frac {b\left(a+b\right)}{1-(a+2b)x-{x}^{2}A_3(x)}}.
\end{align*}

We see that $A_3(x)=A_2(x)$, thus $$H_{n-2}(A_2)=(b(a+b))^{n-2}H_{n-3}(A_2),$$
by which it is easy to deduce that $$H_{n-2}(A_2))=(b(a+b))^{\binom{n-1}{2}}.$$
Hence
\begin{align*}
 H_n(f_1(x))&=b^{\binom{n}{2}}(a+b)^{\binom{n}{2}}.
\end{align*}
The computation for $F_1(x),G_1(x)$ is similar.
\end{proof}

For the case $m\ge 2$, we apply Lemmas \ref{f_m(x))} and \ref{q(x,t)} to obtain the following result.
\begin{thm}\cite{X. K. Chang. X. B. Hu and Y. N. Zhang,J.Cigler}\label{H(f)}
For $m\geq2,$ we have
\begin{align*}
  H_{mn}(f_m(x))&=(-1)^{\binom{m-1}{2} n}b^{n(mn-1)}, \\
  H_{mn+1}(f_m(x))&=(-1)^{\binom{m-1}{2} n}b^{n(mn+1)},
\end{align*}
and $ H_{n}(f_m(x))=0 $ for all other $n$;
\begin{align*}
  H_{mn}(F_m(x))&=(-1)^{\binom{m-1}{2}n}b^{n(mn-1)}, \\
  H_{mn+1}(F_m(x))&=(-1)^{\binom{m-1}{2}n}b^{n(mn+1)},
\end{align*}
and $H_{n}(G_m(x))=0 $ for all other $n$;
\begin{align*}
  H_{mn}(G_m(x))=(-1)^{\binom{m-1}{2} n}2^{mn-1}b^{n(mn-1)},\\
  H_{mn+1}(G_m(x))=(-1)^{\binom{m-1}{2} n}2^{mn}b^{n(mn+1)},
\end{align*}
and $H_{n}(G_m(x))=0 $ for all other $n$.
\end{thm}

\subsection{For $H_n^1(q_m(x,t))$}
We need the notation of generalized Fibonacci numbers: $Fib_n(a,b)=aFib_{n-1}(a,b) + bFib_{n-2}(a,b)$ with initial values $Fib_0(a,b)=0$ and $Fib_1(a,b)=1$.

The case $m\le 2$ has to be considered separately. The formula for $H_n^1(f_1)$ can be obtained by \eqref{f1-f1}. The other formulas
in the following theorem can be obtained by applying Lemma \ref{q(x,t)1} and Theorem \ref{Hq1}.
\begin{thm}\cite{X. K. Chang. X. B. Hu and Y. N. Zhang,J.Cigler}\label{Hf11}
\begin{align*}
  H_n^1(f_1(x))&=b^{\binom{n}{2}}(a+b)^{\binom{n+1}{2}}, \\
  H_n^1(F_1(x))&=b^{\binom{n}{2}}(a+b)^{\binom{n}{2}}((a+b)^n+t\frac{((a+b)^n-b^n)}{a}),\\
  H_n^1(G_1(x))&=2^{n-1}b^{\binom{n}{2}}(a+b)^{\binom{n}{2}}((a+b)^n+b^n).
\end{align*}
\begin{align*}
  H_n^1(f_2(x))&=b^{\binom{n}{2}}Fib_{n+1}(a,-b), \\
  H_n^1(F_2(x))&=b^{\binom{n}{2}}(Fib_{n+1}(a,-b)+tFib_{n}(a,-b)),\\
  H_n^1(G_2(x))&=2^{n-1}b^{\binom{n}{2}}(aFib_{n}(a,-b)-2^{n-1}bFib_{n-1}(a,-b)).
\end{align*}
\end{thm}
\begin{proof}
By \eqref{f1-f1}, we have  $H_{n}^1(f_1)=\frac{H_{n+1}(f_1)}{b^{n}}=b^{\binom{n}{2}}(a+b)^{\binom{n+1}{2}}.$ This is the first equality.

Now by \eqref{f-e-q1} with respect to $f_1$, we have
$$H_{n-1}^{(2)}(f_1(x))=\frac{H_{n}^1(f_1)-b^nH_n(f_1(x))}{ab^{n-1}}=b^{\binom{n-1}{2}}(a+b)^{\binom{n}{2}}\frac{(a+b)^{n}-b^{n}}{a}.$$

To derive the second equality, we apply \eqref{f-e-q1} with respect to $F_1$. This corresponds to $\alpha=a+t, \beta=b$. Then we obtain $$H_{n}^1(F_1)=(a+t)b^{n-1}H_{n-1}^{(2)}(f_1(x))+b^nH_n(f_1(x))=b^{\binom{n}{2}}(a+b)^{\binom{n}{2}}((a+b)^n+t\frac{((a+b)^n-b^n)}{a}).$$

Similarly we can derive the third equality. This corresponds to $\alpha=a, \beta=2b$, and we obtain
$$H_{n}^1(G_1)=a(2b)^{n-1}H_{n-1}^{(2)}(f_1(x))+(2b)^nH_n(f_1(x))=2^{n-1}b^{\binom{n}{2}}(a+b)^{\binom{n}{2}}((a+b)^n+b^n).$$

The proof of the remaining three equalities is almost the same, except that we need to use \eqref{f-e-q2} instead of \eqref{f-e-q1}.
%
%
%
%
%
\end{proof}

For the case $m\ge 3$, we apply Lemmas \ref{f_m(x))} and \ref{q(x,t)1} to obtain the following result.
\begin{thm}\cite{X. K. Chang. X. B. Hu and Y. N. Zhang,J.Cigler}\label{H(G)}
For  $m\geq 3,$ we have
\begin{align*}
H_{mn}^1(f_m(x))&=(-1)^{\binom{m}{2} n }b^{mn^2}, \\
H_{mn+1}^1(f_m(x))&=(-1)^{\binom{m}{2} n}(n+1)a b^{mn^2+2n},\\
H_{mn-1}^1(f_m(x))&=-(-1)^{\binom{m}{2} n} na b^{mn^2-2n},
\end{align*}
and $ H_{n}^1(f_m(x))=0 $ for all other $n$;\\
\begin{align*}
H_{mn}^1(F_m(x))&=(-1)^{\binom{m}{2} n }b^{mn^2}, \\
H_{mn+1}^1(F_m(x))&=(-1)^{\binom{m}{2} n}(t+(n+1)a) b^{mn^2+2n},\\
H_{mn-1}^1(F_m(x))&=-(-1)^{\binom{m}{2} n}(t+ na) b^{mn^2-2n},
\end{align*}
and $ H_{n}^1(F_m(x))=0 $ for all other $n$;\\
\begin{align}
H_{mn}^1(G_m(x))&=(-1)^{\binom{m}{2} n }2^{mn}b^{mn^2}, \nonumber\\
H_{mn+1}^1(G_m(x))&=(-1)^{\binom{m}{2} n}(2n+1)a b^{mn^2+2n},\label{conG1}\\
H_{mn-1}^1(G_m(x))&=-(-1)^{\binom{m}{2} n} (2n-1)2^{mn-2}a b^{mn^2-2n},\nonumber
\end{align}
and $ H_{n}^1(G_m(x))=0 $ for all other $n$.
\end{thm}


\subsection{For $H_n^{(2)}(q_m(x,t))$}
The case $m\le 2$ follows from Theorems \ref{Hq1}, \ref{H(f)} and \ref{Hf11}, and the condensation formula \eqref{e-condensation}.
We obtain the
following result.
\begin{thm}\cite{X. K. Chang. X. B. Hu and Y. N. Zhang,J.Cigler}\label{Hq12}
\begin{align*}
  H_n^{(2)}(f_1(x))&=b^{\binom{n}{2}}(a+b)^{\binom{n+1}{2}}\frac{(a+b)^{n+1}-b^{n+1}}{a}, \\
  H_n^{(2)}(F_1(x))&=b^{\binom{n}{2}}(a+b)^{\binom{n}{2}}\sum_{j=0}^n(a+b)^{n-j}b^{n-j}\left((a+b)^j+t\frac{((a+b)^j-b^j)}{a}\right)^2,\\
  H_n^{(2)}(G_1(x))&=2^{n}b^{\binom{n+1}{2}}(a+b)^{\binom{n+1}{2}}(1+\sum_{j=1}^n\frac{((a+b)^j+b^j)^2}{2(a+b)^jb^j}).
\end{align*}
\begin{align*}
  H_n^{(2)}(f_2(x))&=b^{\binom{n}{2}} \sum _{j=0}^nb^{n-j}(Fib_{j+1}(a,-b))^2,\\
  H_n^{(2)}(F_2(x))&=b^{\binom{n}{2}}\sum _{j=0}^nb^{n-j}\left(Fib_{j+1}(a,-b)+tFib_j(a,-b)\right)^2,\\
  H_n^{(2)}(G_2(x))&=2^{n}b^{\binom{n+1}{2}}\left(1+\sum _{j=1}^n\frac{(Fib_{j}(a,-b)-2^{j-1}bFib_{j-1}(a,-b))^2}{2b^j}\right).
\end{align*}
\end{thm}

The computation for the case $m=3$ is different.
\begin{thm}\cite{X. K. Chang. X. B. Hu and Y. N. Zhang,J.Cigler}
\begin{align*}
  H_{3n}^{(2)}(f_3(x))&=(-1)^{n+1}b^{3n^2+n-1}(a^3\sum_{i=0}^ni^2-(n+1)b), \\
  H_{3n+1}^{(2)}(f_3(x))&=(-1)^{n}(n+1)^2a^2b^{3n^2+3n}, \\
  H_{3n+2}^{(2)}(f_3(x))&=(-1)^{n}b^{3n^2+5n+1}\left(a^3\sum_{i=0}^{n+1}i^2-(n+1)b\right).
\end{align*}
\begin{align*}
  H_{3n}^{(2)}(F_3(x))&=(-1)^{n}b^{3n^2+n-1}\left(a^3\sum_{i=0}^ni^2-(n+1)b+nat(t+(n+1)a)\right), \\
  H_{3n+1}^{(2)}(F_3(x))&=(-1)^{n}(t+(n+1)a)^2b^{3n^2+3n}, \\
  H_{3n+2}^{(2)}(F_3(x))&=(-1)^{n}b^{3n^2+5n+1}\left(a^3\sum_{i=0}^{n+1}i^2-(n+1)b+(n+1)at(t+(n+2)a)\right).
\end{align*}
\begin{align}
  H_{3n}^{(2)}(G_3(x))&=(-1)^{n}\left((2n+1)^22^{3n}b^{3n^2+n}-\binom{2n+1}{3}2^{3n-1}a^3b^{3n^2+n-1}\right),\nonumber \\
  H_{3n+1}^{(2)}(G_3(x))&=(-1)^{n}(2n+1)^22^{3n}a^2b^{3n^2+3n}, \label{conG2}\\
  H_{3n+2}^{(2)}(f_3(x))&=(-1)^{n+1}\left((2n+1)2^{3n+2}b^{3n^2+5n+2}-\binom{2n+3}{3}2^{3n+1}a^3b^{3n^2+5n+1}\right).\nonumber
\end{align}
\end{thm}
\begin{proof}

Let $A_0=\displaystyle\frac{f_3(x)-ax-1}{x^2}$.
Then we obtain
$$A_0(x)=\frac {{a}^{2}b{x}^{3}+2\,ab{x}^{2}+bx+{a}^{2}}{1-ax-2\,b{x}^{3}-2\,ab{x}^{4}-b{x}^{5}A_0(x)}.$$
Repeat application of Proposition \ref{xinu(0,i,ii)} gives $A_1,$ $A_2$, $A_3$ $\cdots$. We have
\begin{align}
 H_{k} (A_{0})=&\left(a^2\right)^kH_{k-1}(A_{1}) .\label{f-A0-A1-}
\end{align}
For $n \geq 0$, we have
\begin{align*}
H_{k} (A_{3n+1})=&\left(-\frac{(a^3\sum_{i=1}^{n+1}i^2-(n+1)b)(a^3\sum_{i=1}^{n}i^2-(n+1)b)}{(n+1)^4a^4}\right)^{k}\!\!H_{k-1} (A_{3n+2}),\\
H_{k-1} (A_{3n+2})=&\left(\frac{(n+1)^2(a^3\sum_{i=1}^{n+1}i^2-(n+2)b)a^2b}{(a^3\sum_{i=1}^{n+1}i^2-(n+1)b)^2}\right)^{k-1}\!\!\!\!\! H_{k-2}(A_{3n+3}),\\
H_{k-2} (A_{3n+3})=&\left(-\frac{(n+2)^2(a^3\sum_{i=1}^{n+1}i^2-(n+1)b)a^2b}{(a^3\sum_{i=1}^{n+1}i^2-(n+2)b)^2}\right)^{k-2}\!\!\!\!\!H_{k-3}(A_{3n+4}).
\end{align*}
Combining the above formulas gives
\begin{align}\label{f-A3n+1}
  H_{k} (A_{3n+1})=&\frac{b^{2k-3}}{a^6}\cdot\frac{(n+2)^{2k-4}}{(n+1)^{2k+2}}\cdot\frac{(a^3\sum_{i=1}^{n}i^2-(n+1)b)^{k}}
  {(a^3\sum_{i=1}^{n+1}i^2-(n+2)b)^{k-3}}H_{k-3}(A_{3n+4}).
\end{align}
Let $k=3n+j$, where $0\leq j\leq 2$. We know that
\begin{align*}
  H_{3n+j} (A_{1})=&(-1)^{n+j}(n+1)^{2j+2}\cdot\frac{b^{3n(n+1)+2jn+j}}{a^{6n}(a^3\sum_{i=1}^{n}i^2-(n+1)b)^{j}}H_{j}(A_{3n+1}).
\end{align*}
Then together with \eqref{f-A0-A1-} we obtain
\begin{align}
  H_{3n+j+1} (A_{0})=&(-1)^{n+j}(n+1)^{2j+2}\cdot\frac{b^{3n(n+1)+2jn+j}a^{2j+2}}{(a^3\sum_{i=1}^{n}i^2-(n+1)b)^{j}}H_{j}(A_{3n+1}).\label{f-A0}
\end{align}
The initial values are
\begin{align*}
 H_{0}(A_{3n+1})=&1, \\
 H_{1}(A_{3n+1})=&-\frac{(a^3\sum_{i=1}^{n+1}i^2-(n+1)b)(a^3\sum_{i=1}^{n}i^2-(n+1)b)}{(n+1)^4a^4},\\
 H_{2}(A_{3n+1})=&\frac{b(a^3\sum_{i=1}^{n}i^2-(n+1)b)^2(a^3\sum_{i=1}^{n+1}i^2-(n+2)b)}{(n+1)^6a^6}.
\end{align*}
Then  $H_n^{(2)}(f_3(x))$ follows by the above initial values and \eqref{f-A0}. Similarly we can compute $H_n^{(2)}(F_3(x))$ and $H_n^{(2)}(G_3(x))$.
\end{proof}

For the case $m\ge 4$, we need the following result, whose proof will be given a bit later.
\begin{lem}\label{Hf-1}
For $m \geq 4$,
\begin{align*}
 H_{mn}(f_m(x)-1)&=-(-1)^{\binom{m-1}{2}n}(n-1)b^{mn^2-n},\\
 H_{mn-1}(f_m(x)-1)&=(-1)^{\binom{m-1}{2}n}n^2a^2b^{mn^2-3n},\\
 H_{mn+1}(f_m(x)-1)&=-(-1)^{\binom{m-1}{2}n}nb^{mn^2+n},\\
 H_{mn+2}(f_m(x)-1)&=-(-1)^{\binom{m-1}{2}n}(n+1)^2a^2b^{mn^2+3n},
\end{align*}
and $H_n(f_m(x)-1)=0$ for all other $n$.
\end{lem}
Then by the above lemma   and the formula
$$H_n(q_m(x,t))=H_{n-1}^2(q_m(x,t))+H_n(q_m(x,t)-1),$$
which is a consequence of \eqref{e-cDxy}, we obtain the following result, which was conjectured in \cite{J.Cigler}, and proved in \cite{X. K. Chang. X. B. Hu and Y. N. Zhang}.
\begin{thm}\cite{X. K. Chang. X. B. Hu and Y. N. Zhang,J.Cigler}\label{f-2}
For  $m\geq 4,$ we have
\begin{align*}
  H_{mn}^{(2)}(f_m(x))&=(-1)^{\binom{m-1}{2} n}(n+1)b^{mn^2+n}, \\
  H_{mn+1}^{(2)}(f_m(x))&=(-1)^{\binom{m-1}{2} n }(n+1)^2 a^2 b^{mn^2+3n},\\
  H_{mn-1}^{(2)}(f_m(x))&=(-1)^{\binom{m-1}{2} n}n b^{mn^2-n}, \\
  H_{mn-2}^{(2)}(f_m(x))&=-(-1)^{\binom{m-1}{2} n }n^2 a^2 b^{mn^2-3n},
\end{align*}
and $H_{n}^{(2)}(f_m(x))=0 $ for all other $n$;\\
\begin{align*}
  H_{mn}^{(2)}(F_m(x))&=(-1)^{\binom{m-1}{2} n}(n+1)b^{mn^2+n}, \\
  H_{mn+1}^{(2)}(F_m(x))&=(-1)^{\binom{m-1}{2} n }(t+(n+1)a)^2 b^{mn^2+3n},\\
  H_{mn-1}^{(2)}(F_m(x))&=(-1)^{\binom{m-1}{2} n}n b^{mn^2-n}, \\
  H_{mn-2}^{(2)}(F_m(x))&=-(-1)^{\binom{m-1}{2} n }(t+na)^2 b^{mn^2-3n},
\end{align*}
and $H_{n}^{(2)}(F_m(x))=0 $ for all other $n$;\\
\begin{align*}
  H_{mn}^{(2)}(G_m(x))&=(-1)^{\binom{m-1}{2} n}(2n+1)2^{mn}b^{mn^2+n}, \\
  H_{mn+1}^{(2)}(G_m(x))&=(-1)^{\binom{m-1}{2} n }(2n+1)^2 2^{mn} a^2 b^{mn^2+3n},\\
  H_{mn-1}^{(2)}(G_m(x))&=(-1)^{\binom{m-1}{2} n}(2n-1)2^{mn-1} b^{mn^2-n}, \\
  H_{mn-2}^{(2)}(G_m(x))&=-(-1)^{\binom{m-1}{2} n }(2n-1)^2 2^{mn-3} a^2 b^{mn^2-3n},
\end{align*}
and $ H_{n}^{(2)}(G_m(x))=0 $ for all other $n$.
\end{thm}

\begin{proof}[Proof of Lemma \ref{Hf-1}]
Let $A_0=f_m(x)-1$, then we have $$A_0(x)=\frac{bx^m+ax}{1-ax-2bx^m-bx^mA_0(x)}.$$

By Proposition \ref{xinu(0,i,ii)}, we get $$A_1(x)=\frac{bx^{m-4}}{a-a^2x+2bx^{m-1}-x^3(bx^{m-1}+a)A_1(x)},$$ and
\begin{gather}
  H_n(A_0(x))=(-1)^{\binom{2}{2}}a^nH_{n-2}(A_1(x)).\label{f-1}
\end{gather}

In the same way, we obtain $A_2(x),A_3(x)\cdots$, then we list the results as follows,
 $$A_2(x)=\frac{bx(bx^{m-1}+2a^2x-a)}{a(a-a^2x+2bx^{m-1}-ax^{m-2}A_2(x))},$$
 $$A_3(x)=\frac{a^2(4a^2bx^m+4abx^{m-1}+bx^{m-2}+8a^4x^2+4a^3x+4a^2)}{a(-1+ax+8a^3x^3+2bx^{m}+4abx^{m+1})-x^3(bx^{m-1}+2a^2x-a)A_3(x)},$$

\begin{align*}
  A_4(x) =&(bx^{m-4})/(4a^2-4a^3x+2bx^{m-2}+4abx^{m-1}-x^2(4a^2bx^m+4abx^{m-1}+bx^{m-2}\\
  &+8a^4x^2+4a^3x+4a^2)A_4(x)),
\end{align*}
and
\begin{align}
H_n(A_1(x))&=(-1)^{\binom{m-3}{2}}\left(\frac{b}{a}\right)^n\!\!H_{n-m+3}(A_2(x)),\nonumber \\
H_n(A_2(x))&=(-1)^{\binom{2}{2}}\left(-\frac{b}{a}\right)^n\!\!H_{n-2}(A_3(x)),\label{f-1,1}\\
H_n(A_3(x))&=\left(-4a^3\right)^n\!\!H_{n-1}(A_4(x)).\nonumber
\end{align}
By \eqref{f-1} and \eqref{f-1,1},  we know that
\begin{align}
H_n(A_0)=&(-1)^{\binom{m-1}{2}}(2a)^{2n-2m-2}b^{2n-m-1}H_{n-m-2}(A_4).\label{A0A4}
\end{align}

Further transformations suggest us to define
$$A_{4p}(x)= {\frac {{x}^{m-4}}{u(x)-{x}^{2} v(x)A_{4p}(x)}},$$
with
\begin{align*}
\begin{array}{rl}
v(x)=&\left( p+1 \right) ^{2}{a}^{2}b{x}^{m}+2\,p
 \left( p+1 \right) ab{x}^{m-1}+{p}^{2}b{x}^{9}+p \left( p+1 \right) ^{
3}{a}^{4}{x}^{2}+ \left( p+1 \right) ^{2}{a}^{3}x\\
&+p \left( p+1\right) ^{2}{a}^{2}/(b{p}^{2}) ,
\end{array}
\end{align*}
$$u(x)={\frac {2\,p \left( p+1 \right) ab{x}^{m-1}+2\,{p}^{2}b{x}^{m-2}-p \left( p+1 \right) ^{2}{a}^{3}x+p \left( p+1 \right) ^{2}{a}^{2}}{b{p}^{2}}}.$$

Then our algorithm produces
\begin{align*}
H_n(A_{4p})&=(-1)^{\binom{m-3}{2}}\left(\frac{pb}{(p+1)^2a^2}\right)^n\!\!H_{n-m+3}(A_{4p+1})\\
  H_n(A_{4p+1})&=\left( -{\frac {p b}{ \left( p+1 \right) ^{2}{a}^{2}}}\right)^n\!\!H_{n-1}(A_{4p+2}),\\
  H_n(A_{4p+2})& =\left( -{\frac {(p+1)^3{a}^{2}}{  p^{2}}}\right)^n\!\!H_{n-1}(A_{4p+3}),\\
  H_n(A_{4p+3})&=\left( {\frac {p(p+2)^2{a}^{2}}{ \left( p+1 \right) ^{2}}}\right)^n\!\!H_{n-1}(A_{4p+4}).
\end{align*}
Combining the above formulas gives
\begin{gather*}
 H_n(A_{4p})=-(-1)^{\binom{m-1}{2}}p^n(p+1)^{-3n+m-2}(p+2)^{2n-2m+2}a^{-2m}b^{2n-m+3}H_{n-m}(A_{4(p+1)}).
\end{gather*}
Using this recursion together with some initial conditions we obtain
\begin{align*}
H_{mn}(A_4)=&-(-1)^{\binom{m-1}{2}n}(n+2)^22^{-2mn-2}a^{-2mn}b^{mn^2+3n}\\
H_{mn-1}(A_4)=&-(-1)^{\binom{m-1}{2}n}(n+1)(2a)^{-2mn}b^{mn^2+n}\\
H_{mn-2}(A_4)=&-(-1)^{\binom{m-1}{2}n}n(2a)^{-2mn+2}b^{mn^2-n} \\
H_{mn-3}(A_4)=&(-1)^{\binom{m-1}{2}n}(n+1)^22^{-2mn+4}a^{-2mn+6}b^{mn^2-3n},
\end{align*}
and $H_n(A_4)=0$ for other $n$.

The proof is thus completed by applying equation
\eqref{A0A4}.
\end{proof}

\section{Two periodic continued fractions \label{sec:periodic}}
Let us focus on lattice paths from $(0,0)$ to $(n,0)$ that never go below the horizontal axis. In this section, we give two examples of periodic continued fractions.

Our first example   has step set $\{ U=(1,1), D=(1,-1), H_1=(1,0), H_2=(2,0)\}$, where the $H_1$ step is weighted by $t+1$, the $H_2$ step is weighted by $t$, and all other steps is weighted by $1$.
Then the generating function $F(x)$ of such lattice paths satisfies the following quadratic functional equation:
\[ F(x)=1+(t+1)xF(x)+tx^2F(x)+x^2F(x)^2.\]
\begin{thm}
Let $F(x)$ be as above with $t$ an indeterminate.  Then we have
$$H_{3n}(F(x))=(1+t)^{3n^2},\ H_{3n+1}(F(x))=(1+t)^{3n^2+2n},\  H_{3n+2}(F(x))=(1+t)^{3n^2+4n+1}.$$
\end{thm}
\begin{proof}
Let's denote $F_0(x)=F(x)$. Then the functional equation can be rewritten as
\[F_0(x)=\frac{1}{1-(1+t)x-tx^2-x^2F_0(x)}.\]
We will repeatedly apply $\tau$ in Proposition \ref{xinu(0,i,ii)} to obtain a recursion. We will carry out the details in this computation.
Note that all the $u(x)$ appears are in fact polynomials, and hence the decompositions $u(x)=u_L(x)+x^{d+2}u_H(x)$ are simple.

Apply Proposition \ref{xinu(0,i,ii)}  to obtain $F_1(x)=\tau(F_0(x))$, where $u(0)=1, d=0$. We have
\[H_n(F_0)=H_{n-1}(F_1),\ \  \ F_1(x)=\frac{(1+t)(tx-1)}{(tx-1)(1-x)+x^2F_1(x)}.\]

Apply Proposition \ref{xinu(0,i,ii)} to obtain $F_2=\tau(F_1)$, where $u(0)^{-1}=t+1,\ d=0$. We have
\[H_{n-1}(F_1)=(1+t)^{n-1}H_{n-1}(F_2),\ F_2(x)=\frac{tx-1}{(tx-1)(1-x)+(1+t)x^2F_2(x)}.\]

Apply Proposition \ref{xinu(0,i,ii)} to obtain $F_3=\tau(F_2)$, where $u(0)=1, d=0$.  We have
\[H_{n-1}(F_2)=H_{n-2}(F_3),\ \  \ F_3(x)=\frac{1+t}{(tx-1)(-1+x+x^2F_3(x))}.\]

Apply Proposition \ref{xinu(0,i,ii)} to obtain $F_4=\tau(F_3)$, where $u(0)^{-1}=t+1,\ d=0$. We have
\[H_{n-2}(F_3)=(1+t)^{n-2}H_{n-2}(F_4),\ \ F_4(x)=\frac{1}{(tx-1)(-1+x+x^2(1+t)F_4(x))}.\]

Apply Proposition \ref{xinu(0,i,ii)} to obtain $F_5=\tau(F_4)$, where $u(0)=1, d=0$.  We have
\[H_{n-2}(F_4)=H_{n-3}(F_5),\ \  \ F_5(x)=\frac{1}{1-(1+t)x-tx^2-x^2F_5(x)}.\]
Now we see that $F_5(x)=F_0(x).$ Summarizing the above results gives the recursion
$H_n(F_0)=(1+t)^{2n-3}H_{n-3}(F_0)$. Together with the initial condition $H_0(F_0)=1,\ H_1(F_0)=1,\ H_2(F_0(x))=1+t,$
we conclude that, for example,
 \[H_{3n}(F_0)=(1+t)^{6n-3}H_{3(n-1)}(F_0)=\cdots =(1+t)^{2(3n+3(n-1)+\cdots+6+3)-3n}H_0(F_0)=(1+t)^{3n^2}.\]
The cases for $H_{3n+1}(F_0)$ and $H_{3n+2}(F_0)$ can be computed similarly.
\end{proof}

Our second example  has step set $\{ U=(1,1), D=(1,-1), H_2=(2,0), H_3=(3,0)\}$, where the $H_2$ step is weighted by $t$, the $H_3$ step is weighted by $t+1$, and all other steps is weighted by $1$.
Then the generating function $F(x)$ of such lattice paths satisfies the following quadratic functional equation:
\[ F(x)=1+tx^2F(x)+(t+1)x^3F(x)+ x^2F(x)^2.\]
\begin{thm}
Let $F(x)$ be as above with $t$ an indeterminate.  Then we have
\begin{align*}
   & H_{7n}(F(x))=(-1)^n(1+t)^{14n^2},\\
   & H_{7n+1}(F(x))=(-1)^n(1+t)^{14n^2+4n}, \\
   & H_{7n+2}(F(x))=(-1)^n(1+t)^{14n^2+8n+1},\\
   & H_{7n+3}(F(x))=H_{7n+4}(F(x))=0,\\
   & H_{7n+5}(F(x))=(-1)^{n+1}(1+t)^{14n^2+20n+7},\\
   &H_{7n+6}(F(x))=(-1)^{n+1}(1+t)^{14n^2+24n+10}.
\end{align*}
\end{thm}
\begin{proof}
Let $F_0(x)=F(x)$. Then $F_0(x)$ is determined by
\[F_0(x)=\frac{1}{1-tx^2-(t+1)x^3-x^2F_0(x)}.\]

Repeat application of Proposition \ref{xinu(0,i,ii)} gives
\begin{align*}
   &H_n(F_0)=H_{n-1}(F_1),& F_1(x)=\frac{(1+x)(t+1)}{(x+1)(1-x+(1+t)x^2)-x^2F_1(x)}.\\
   & H_{n-1}(F_1)=(1+t)^{n-1}H_{n-2}(F_2), & F_2(x)=\frac{x^2(1+t)}{(1+x)(1-x-(1+t)x^2-x^2F_2(x))}.\\
   &H_{n-2}(F_2)=(-1)^{\binom{3}{2}}(1+t)^{n-2}H_{n-5}(F_3), &F_3(x)=\frac{(x+1)(1+t)}{(1+x)(1-x-(1+t)x^2)-x^4F_3(x)}.\\
   & H_{n-5}(F_3)=(1+t)^{n-5}H_{n-6}(F_4), & F_4(x)=\frac{(1+t)}{(1+x)(1-x+(1+t)x^2-x^2F_4(x))}.\\
   &H_{n-6}(F_4)=(1+t)^{n-6}H_{n-7}(F_5), & F_5(x)=\frac{1}{1-tx^2-(1+t)x^3-x^2F_5(x).}\qquad \
\end{align*}

Now we see that $F_5(x)=F_0(x)$. Then we obtain $H_n(F_0)=-(1+t)^{4n-14}H_{n-7}(F_0)$.

The initial conditions are $H_0(F_0)=H_1(F_0)=1$, $H_2(F_0)=1+t,$ $H_3(F_0)=H_4(F_0)=0$, $H_5(F_0)=-(1+t)^7$, $H_6(F_0)=-(1+t)^{10}$.  Therefore
\[H_{7n}(F_0)=-(1+t)^{28n-14}H_{7(n-1)}(F_0),\]
Recursively, we get\[H_{7n}(F_0)=(-1)^n(1+t)^{4(7n+7(n-1)+\cdots+14+7)-14n}H_0(F_0)=(-1)^n(1+t)^{14n^2}.\]
The other six cases follow similarly.
\end{proof}

\section{Four classes of Hankel determinants \label{sec:four-classes}}

In our setup, it is convenient to use a functional equation to define a generating function. In this section,
we consider generated function $F(x)$ determined by
\begin{align}
  F(x)=\dfrac{x^m}{1-ax^q-bx^\ell F(x)}. \label{e-F-mql}
\end{align}
The step set of the corresponding paths is $\{(1,1),(q,0), (\ell-1,-1)\}$. Note that $m,q,\ell$ are integer parameters.

Computer experiment shows that the Hankel determinants $H_n(F(x))$ do not always have nice formulas. However, we find nice formulas
for four classes of $m,q,\ell$, and it seems that there is no more nice formulas.
As far as we know, these results have not appeared in the literature.

\begin{thm}
Let $F(x)$ be determined by \eqref{e-F-mql}. If $m,\ \ell,\ q$ are non-negative integers satisfying $m+1\geq q>0$ and $\ell\geq q+1$, then we have
\begin{align*}
 H_{(n+1)(m+\ell)}(F)&=(-1)^{(n+1)\left[\binom{m+1}{2}+\binom{\ell-1}{2}\right] }b^{(n+1)(n\ell+nm+\ell-1)}, \\
 H_{n(m+\ell)+m+1}(F)&=(-1)^{(n+1)\binom{m+1}{2}+n\binom{\ell-1}{2} }b^{n^2(\ell+m)+n(m+1)},
\end{align*}
and $H_n(F(x))=0$ for all other situations.
\end{thm}
\begin{proof}
Apply Proposition \ref{xinu(0,i,ii)} to $F_0(x):=F(x)$. We obtain
\[H_n(F)=(-1)^{\binom{m+1}{2}}H_{n-m-1}(F_1),\ \qquad F_1(x)=\frac{bx^{\ell-2}}{1-ax^q-x^{m+2}F_1(x)},\]
Repeated application of Proposition \ref{xinu(0,i,ii)} gives
\begin{align*}
  H_{n-m-1}(F_1)&=(-1)^{\binom{\ell-1}{2}}b^{n-m-1}H_{n-m-\ell}(F_2), &F_2(x)=\frac{bx^m}{1-ax^q-x^\ell F_2(x)}.\quad \ \\
  H_{n-m-\ell}(F_2)&=(-1)^{\binom{m+1}{2}}b^{n-m-\ell}H_{n-2m-\ell-1}(F_3),  &F_3(x)=\frac{bx^{\ell-2}}{1-ax^q-x^{m+2}F_3(x)}.
\end{align*}
Note that in application of each $\tau$, we have $u_L(x)=1-ax^q$ and $u_H(x)=0$, due to the conditions on $m,\ell,$ and $q$. This allows us
to do the transformations for symbolic $m,\ell,q$.

Now we see that $F_3(x)=F_1(x),$ and the Hankel determinants of $F_1(x)$ have a period of $m+\ell$. Therefore we obtain
\[H_{n-m-1}(F_1)=(-1)^{\binom{m+1}{2}+\binom{\ell-1}{2}}b^{2n-2m-\ell-1}H_{n-2m-\ell-1}(F_1).\]
\begin{align}\label{eq411}
  H_n(F)&=(-1)^{\binom{m+1}{2}}H_{n-m-1}(F_1)=(-1)^{\binom{\ell-1}{2}}b^{2n-2m-\ell-1}H_{n-2m-\ell-1}(F_1).
\end{align}
Let $n-m-1=(\ell+m)p+i$ be a non-negative integer. We then deduce that
\begin{equation*}
 H_{n-m-1}(F_1)=(-1)^{p[\binom{m+1}{2}+\binom{\ell-1}{2}]}b^{p^2(\ell+m)+p(m+2i+1)}H_i(F_1).
\end{equation*}
The initial values are
$$H_0(F_1)=1,\  H_{\ell-1}(F_1)=(-1)^{\binom{\ell-1}{2}}\ b^{\ell-1}, \text{ and } H_i(F_1)=0 \text{ for } 1\leq i\leq \ell-2\text{ and }\ell\leq i<\ell+m.$$
The theorem then follows by the above initial values and \eqref{eq411}.
\end{proof}
\begin{thm}
For $m+\ell=q$, $\ell\geq 2$, $m$, $\ell$, $n$, $q$ are all non-negative integers, we have
\begin{align*}
  H_{(\ell+m)n}(F)&=(-1)^{\binom{m+1}{2}n+\binom{\ell-1}{2}n}(a+b)^{\frac{(\ell+m)(n-1)n}{2}+(\ell-1)n}b^{\frac{(\ell+m)(n-1)n}{2}}, \\
  H_{(\ell+m)n+m+1}(F)&=(-1)^{\binom{m+1}{2}(n+1)+\binom{\ell-1}{2}n}(a+b)^{\frac{(\ell+m)(n+1)n}{2}}b^{\frac{(\ell+m)(n+1)n}{2}-(\ell-1)n},
\end{align*}
and $H_n(F)=0$ for all other situations.
\end{thm}
\begin{proof}
Apply Proposition \ref{xinu(0,i,ii)} to $F_0(x):=F(x)$, we obtain
$$H_n(F_0)=(-1)^{\binom{m+1}{2}}H_{n-(m+1)}(F_1),\ \qquad F_1(x)=\frac{(a+b)x^{\ell-2}}{1+ax^q-x^{m+2}F_1(x)}.$$
Repeated application of Proposition \ref{xinu(0,i,ii)} gives
\begin{align*}
  H_n(F_1)&=(-1)^{\binom{\ell-1}{2}}(a+b)^n H_{n-(\ell-1)}(F_2),& F_2(x)=\frac{bx^m}{1-ax^q-x^\ell F_2(x)}. \quad \ \\
 H_n(F_2)&=(-1)^{\binom{m+1}{2}}b^nH_{n-(m+1)}(F_1), & F_3(x)=\frac{(a+b)x^{\ell-2}}{1+ax^q-x^{m+2}F_3(x)}.
\end{align*}
Now we see that $F_3(x)=F_1(x)$, and the Hankel determinants of $F_1(x)$ have a period of $\ell+m$. Therefore we obtain $H_n(F_1)=(-1)^{\binom{\ell-1}{2}+\binom{m+1}{2}}(a+b)^nb^{n-\ell+1}H_{n-\ell-m}(F_1).$
 Using this recursion together with some initial conditions we obtain
\begin{align*}
  H_{(\ell+m)n}(F_1) &=(-1)^{\binom{\ell-1}{2}n+\binom{m+1}{2}n}(a+b)^{\frac{(\ell+m)(n+1)n}{2}}b^{\frac{(\ell+m)(n+1)n}{2}-(\ell-1)n}, \\
  H_{(\ell+m)n+\ell-1}(F_1)&=(-1)^{\binom{m+1}{2}n+\binom{\ell-1}{2}(n+1)}(a+b)^{\frac{(\ell+m)(n+1)n}{2}+(\ell-1)(n+1)}b^{\frac{(\ell+m)(n+1)n}{2}},\\
  H_n(F_1)&=0 \text{ for all other situations.}
\end{align*}
Then we can get $H_n(F)$.
\end{proof}
\begin{thm}\label{t-for-Catalan}
Let $F(x)$ be determined by \eqref{e-F-mql}. If $m,\ \ell, \ q$ are non-negative integers satisfying
$i=q-m\geq2$ and $ \ell\geq q+2i-1$, then we have
\begin{align*}
H_{n(\ell+m)+m+i}(F)&=(-1)^{(n+1)\binom{m+1}{2}+n\binom{\ell-2i+1}{2}+\binom{i-1}{2}-n(i+3)}b^{n^2(\ell+m)+n(2i+m-1)}[(n+1)a]^{i-1},\\
  H_{(n+1)(\ell+m)-i+1}(F)&=  (-1)^{(n+1)[\binom{\ell-2i+1}{2}+\binom{m+1}{2}]+\binom{i-1}{2}-n(i+3)}b^{(n+1)(n\ell+nm+\ell-2i+1)}
    [(n+1)a]^{i-1},\\
H_{(n+1)(\ell+m)}(F)&=(-1)^{(n+1)[\binom{\ell-2i+1}{2}+\binom{m+1}{2}]-n(i+3)+i-1}b^{(n+1)(n\ell+nm+\ell-1)},\\
H_{n(\ell+m)+m+1}(F)&=(-1)^{n\binom{\ell-2i+1}{2}+(n+1)\binom{m+1}{2}-n(i+3)}b^{n(n+1)(\ell+m)-n(\ell-1)},
\end{align*}
and $H_n(F)=0$ for all other situations.
\end{thm}
\begin{proof}
We apply Proposition \ref{xinu(0,i,ii)} to $F_0(x):=F(x)$, and repeat the transformation $\tau$. The conditions on $m,\ell,q$ allow us
to compute symbolically and we will give detailed decomposition of the corresponding $u(x)$.

Computer experiment suggests us to define
\[F_{4p+1}=\frac{x^{i-2}[p(p+1)a^2x^q+bx^{\ell+m-q}+(p+1)a]}{1+(2p+1)ax^q-x^{q-(i-2)}F_{4p+1}}.\]
Then the results may be summarized as follows.
\[H_n(F_0)=(-1)^{\binom{m+1}{2}}H_{n-m-1}(F_1),\]
and for $p\ge 0$, by the conditions on $m,\ell,q$, the process is as follows.

Apply Proposition \ref{xinu(0,i,ii)} to obtain $F_{4p+2}=\tau (F_{4p+1})$.
Firstly, $d=i-2$. Next we need to decompose $u(x)$ with respect to $d$.
We expand $u(x)$ as a power series and focus on (by displaying) those terms  with small exponents:
$$u(x)=\frac{1+(2p+1)ax^q}{p(p+1)a^2x^q+bx^{\ell+m-q}+(p+1)a}
=\frac{1}{(p+1)a}+\frac{2p+1}{(p+1)}x^q-\frac{b}{(p+1)a}x^{\ell+m-q}+\cdots,$$
By the conditions on $m,\ell, q$, we have
$q=i+m>i-1=d+1$, and $\ell+m-q \geq 2i-1>i-1=d+1$. Thus
$u_L(x)=\displaystyle\frac{1}{(p+1)a}$ is simple, and
$ u_H(x)=\displaystyle\frac{(p+1)^2a^2x^m-bx^{\ell-2i}}
{(p(p+1)a^2x^q+bx^{\ell+m-q}+(p+1)a)(p+1)a}.$
Then by $u(0)^{-1}= (p+1)a$, we obtain
 \begin{align*}
  &H_{n-m-1}(F_{4p+1})=(-1)^{\binom{i-1}{2}}\left((p+1)a\right)^{n-m-1}H_{n-m-1-(i-1)}(F_{4p+2}),\\
  &F_{4p+2}=\frac{bx^{\ell-2i}}{(p+1)a-(p+1)a^2x^q+2bx^{\ell+m-q}-((p+1)a+p(p+1)a^2x^q+bx^{\ell+m-q})x^iF_{4p+2}}.
\end{align*}

Apply Proposition  \ref{xinu(0,i,ii)} to obtain $F_{4p+3}=\tau (F_{4p+2})$. This time $d=\ell-2i$ and $u(x)$ is indeed a polynomial: $$u(x)=\displaystyle\frac{(p+1)a-(p+1)a^2x^q+2bx^{\ell+m-q}}{b}.$$
Now $q\leq \ell -2i+1 =d+1$, and $\ell+m-q=\ell-i> d+1$. It then follows that $u_L(x)=\displaystyle\frac{(p+1)a-(p+1)a^2x^q}{b},$
$u_H(x)=2x^{i-2}.$
Then by $u(0)^{-1}=\dfrac{b}{(p+1)a}$, we obtain
\begin{align*}
&H_{n-m-1-(i-1)}(F_{4p+2})=(-1)^{\binom{\ell-2i+1}{2}}\left(\frac{b}{(p+1)a}\right)^{n-m-1-(i-1)}\!\!H_{n-m-1-(\ell-i)}(F_{4p+3}),\\
&F_{4p+3}=\frac{x^{i-2}b[(p+1)(p+2)a^2x^q+bx^{\ell+m-q}-a(p+1)]}{(p+1)a[(p+1)a-2bx^{\ell+m-q}-(p+1)a^2x^q-(p+1)ax^{\ell+m-q-(i-2)}F_{4p+3}]}.
\end{align*}

Apply Proposition  \ref{xinu(0,i,ii)} to obtain $F_{4p+4}=\tau (F_{4p+3})$. Clearly $d=i-2$. Now it is not simple to expand  $$u(x)=\displaystyle\frac{a(p+1)[(p+1)a-2bx^{\ell+m-q}-(p+1)a^2x^q]}{b[(p+1)(p+2)a^2x^q+bx^{\ell+m-q}-a(p+1)]},$$
as a power series and focus on small degree terms. We claim that
$u_L(x)=\displaystyle-\frac{(p+1)a}{b},$
$u_H(x)=\displaystyle\frac{a(p+1)[(p+1)^2a^2x^m-
bx^{\ell-2i}]}{b[(p+1)(p+2)a^2x^q+bx^{\ell+m-q}-a(p+1)]}.$ To see this, one verify that $u(x)=u_L(x)+x^{d+2} u_H(x)$ and check that $u_H(x)$ is a power series by the facts $q>0$, $\ell+m-q=\ell -i\geq q+i-1>0$, $m\geq 0$, $\ell-2i\geq q-1>0$. Finally, by $u(0)^{-1}=-\dfrac{b}{(p+1)a}$, we obtain
\begin{align*}
 & H_{n-m-1-(\ell-i)}(F_{4p+3})=(-1)^{\binom{i-1}{2}}\left(-\frac{b}{(p+1)a}\right)^{n-m-1-(\ell-i)}\!\!H_{n-m-1-(\ell-1)}(F_{4p+4}),\\
 & F_{4p+4}=\frac{-(p+1)^2a^2x^m}{(p+1)a[1-(2p+3)ax^q]+[(p+1)(p+2)a^2x^q+bx^{\ell+m-q}-a(p+1)]x^iF_{4p+4}}.
\end{align*}

Apply Proposition  \ref{xinu(0,i,ii)} to obtain $F_{4p+5}=\tau (F_{4p+4})$. Clearly $d=m$. Since\\
 $u(x)=\displaystyle-\frac{1-(2p+3)ax^q}{a(p+1)}$ is a polynomial
and $q>m+1=d+1$,
we have\\
 $u_L(x)=\displaystyle-\frac{1}{(p+1)a}, \ u_H(x)=\displaystyle\frac{(2p+3)x^{i-2}}{p+1}$. By $u(0)^{-1}=-(p+1)a$, we obtain
\begin{align*}
  H_{n-m-1-(\ell-1)}(F_{4p+4})&=(-1)^{\binom{m+1}{2}}(-(p+1)a)^{n-m-1-(\ell-1)}H_{n-m-1-(\ell+m)}(F_{4p+5}).
\end{align*}
By combining the above formulas we obtain
\[H_{n-m-1}(F_{4p+1})=(-1)^{\binom{\ell-2i+1}{2}+\binom{m+1}{2}-(i+3)}b^{2(n-m-1)-(\ell-1)}H_{n-m-1-(\ell+m)}(F_{4(p+1)+1}).\]
Let $n-m-1=p(\ell+m)+j$, where $0\leq j< \ell+m$. We then deduce that
\begin{align}
  H_{n-m-1}(F_1)&=(-1)^{p[\binom{\ell-2i+1}{2}+\binom{m+1}{2}-(i+3)]}b^{p(p+1)(\ell+m)+p(2j-\ell+1)}H_j(F_{4p+1}),\label{e-q-mpl-F1w}\\
  H_n(F)&=(-1)^{p[\binom{\ell-2i+1}{2}-(i+3)]+(p+1)\binom{m+1}{2}}b^{p(p+1)(\ell+m)+p(2j-\ell+1)}H_j(F_{4p+1}).\label{e-q-mpl-Fw}
\end{align}
The initial values are
\begin{align*}
H_{i-1}(F_{4p+1})&=(-1)^{\binom{i-1}{2}}[(p+1)a]^{i-1},\\
H_{\ell-i}(F_{4p+1})&=(-1)^{\binom{i-1}{2}+\binom{\ell-2i+1}{2}}b^{\ell-2i+1}[(p+1)a]^{i-1},\\
H_{\ell-1}(F_{4p+1})&=(-1)^{\binom{\ell-2i+1}{2}+i-1}b^{\ell-1},\\
H_0(F_{4p+1})&=1,\\
H_j(F_{4p+1})&=0, \text{ for all other } j.
\end{align*}
The theorem then follows by the above initial values, \eqref{e-q-mpl-Fw} and \eqref{e-q-mpl-F1w} .
\end{proof}

\begin{thm}\label{thm411}
Let $F(x)$ be determined by \eqref{e-F-mql}. If $m,\ \ell,\ i, \ q$ are non-negative integers satisfying $\ell=q-i \geq 2,\ m\geq q+2i+1$, then we have
\begin{align*}
 H_{m+1}(F)&=(-1)^{\binom{m+1}{2}}, \\
 H_{(n+1)(m+\ell)}(F)&=(-1)^{n[\binom{m-2i-1}{2}+(i-1)]+(n+1)\binom{\ell-1}{2}+\binom{m+1}{2}}b^{(n+1)^2(m+\ell)-(n+1)(m+1)},\\
 H_{(n+1)(m+\ell)+i+1}(F)&= (-1)^{n[\binom{m-2i-1}{2}+(i-1)]+(n+1)\binom{\ell-1}{2}+\binom{i+1}{2}+\binom{m+1}{2}}\\
   &\qquad b^{(n+1)^2(m+\ell)+n(2i-m+1)+i-m} [(n+1)a]^{i+1},\\
 H_{(n+1)(m+\ell)+m-i}(F)  =&(-1)^{(n+1)[\binom{m-2i-1}{2}+\binom{\ell-1}{2}]+n(i-1)+\binom{i+1}{2}+\binom{m+1}{2}}\\
   &\qquad b^{(n+1)^2(\ell+m)+n(m-2i-1)+m-3i-2}[(n+1)a]^{i+1},\\
  H_{n(m+\ell)+m+1}(F)=& (-1)^{(n+1)[\binom{m-2i-1}{2}+\binom{\ell-1}{2}+i+1]+\binom{m+1}{2}}b^{(n+1)^2(m+\ell)+(n+1)(m+1)},
\end{align*}
and $H_{n}(F)=0$ for all other situations.
\end{thm}
\begin{proof}
Apply Proposition \ref{xinu(0,i,ii)} to $F_0(x):=F(x)$ and repeat the transformation $\tau$. The conditions on $m,\ell,q$ allow us
to compute symbolically. If we define
\[F_{4p+2}=\frac{x^i(bx^{m-i}+n(n-1)a^2x^q+na)}{1+(2n-1)ax^q-x^\ell F_{4p+2}},\]
then the results can be summarized as follows.
\begin{gather}\label{e-q-f0f1w}
  H_n(F_0)=(-1)^{\binom{m+1}{2}}H_{n-m-1}(F_1),\quad  H_{n-m-1}(F_1)=(-1)^{\binom{\ell-1}{2}}b^{n-m-1}H_{n-m-\ell}(F_2).
\end{gather}
and
\begin{align*}
H_{n-m-\ell}(F_{4p+2})&=(-1)^{\binom{i+1}{2}}\left[(p+1)a\right]^{n-m-\ell}H_{n-m-\ell-(i+1)}(F_{4p+3}), \\
 H_{n-m-\ell-(i+1)}(F_{4p+3})&=(-1)^{\binom{m-2i-1}{2}}\left(\frac{b}{(p+1)a}\right)^{n-m-l-(i+1)}\!\!H_{n-m-\ell-(m-i)}(F_{4p+4}),\\
 H_{n-m-\ell-(m-i)}(F_{4p+4})&=(-1)^{\binom{i+1}{2}}\left(-\frac{b}{(p+1)a}\right)^{n-m-\ell-(m-i)}\!\!H_{n-m-\ell-(m+1)}(F_{4p+5}),\\
H_{n-m-\ell-(m+1)}(F_{4p+5})&=(-1)^{\binom{\ell-1}{2}}\left[-(p+1)a\right]^{n-m-\ell-(m+1)}\!H_{n-m-\ell-(m+\ell)}(F_{4p+6}).
\end{align*}

It can be seen that the determinant of $F_{4p+2}$ have a period of $(\ell+m)$,
\[H_{n-m-\ell}(F_{4p+2})=(-1)^{\binom{m-2i-1}{2}+\binom{\ell-1}{2}+i-1}b^{2(n-m-\ell)-(m+1)}H_{n-m-\ell-(\ell+m)}(F_{4p+6}).\]
Let $n-m-\ell=p(\ell+m)+j$, where $0\leq j< \ell+m$. We then deduce that
\begin{align}
  H_{n-m-\ell}(F_2)&=(-1)^{p[\binom{m-2i-1}{2}+\binom{\ell-1}{2}+i-1]}b^{p(p+1)(m+\ell)+p(2j-m-1)}H_j(F_{4p+2}).\label{e-q-mql-2w}
\end{align}
The initial values are
\begin{align*}
  H_0(F_{4p+2})&=1, \\
  H_{i+1}(F_{4p+2})&=(-1)^{\binom{i+1}{2}}[p+1)a]^{i+1},\\
  H_{m-i}(F_{4p+2})&=(-1)^{\binom{i+1}{2}+\binom{m-2i-1}{2}}b^{m-2i-1}[(p+1)a]^{i+1},\\
  H_{m+1}(F_{4p+2})&=(-1)^{\binom{m-2i-1}{2}+i+1}b^{m+1},\\
  H_j(F_{4p+2})&=0 \text{ for all other $j$}.
\end{align*}
The theorem then follows by the above initial values, \eqref{e-q-f0f1w} and \eqref{e-q-mql-2w} .
\end{proof}

\section{Concluding remark}

We have seen that shifted periodic continued fractions appear in our proof of Cigler's Hankel determinant conjectures,
and also appear in Hankel determinants of some path counting numbers. We believe that they will be discovered to appear in more situations.
Basically, nice Hankel determinants of quadratic generating functions may result in shifted periodic continued fractions.

In \cite{J.Cigler-Catalan}, Cigler considered Hankel determinants of $C(x)^r$, which is called the convolution power of the Catalan generating functions.
By using the transformation $\tau$ we have discovered some shifted periodic continued fractions for $H_n(C(x)^r)$ with $r\le 9$, and seems the computation for much larger $r$ is possible. These will appear
in an upcoming paper \cite{Wang-Xin-Catalan-power}.

Another possible research project is to study the nice Hankel determinants in \cite{J.Cigler-C.Krattenthaler} and \cite{C.Krattenthaler-D.Yaqubi},
where the generating functions are quadratic. Can we generalize the step set and still obtain nice Hankel determinants?

Currently, our maple package take the functional equation of $F$ as input, and output the functional equation of $\tau(F)$. To discover
shifted periodic continued fractions, we have to guess manually. One of our next projects is to teach computer to do all these steps automatically.

\medskip
\noindent
\textbf{Acknowledgments}: Part of this work appears in the M.S. thesis \cite{MeiMei} of the third named author. We are grateful to the anonymous referee for suggestions on improving the text and correcting misprints.


\begin{thebibliography}{12}
\bibliographystyle{alpha}
\bibitem{M. Aigner}  M. Aigner. Motzkin numbers. European Journal of Combinatorics, 19(6):663--675, 1998.

\bibitem{D.M. Bressoud and A.M. Americaof.}  D. M. Bressoud. Proofs and confirmations: The story of the alternating sign matrix conjecture. Cambridge University Press Cambridge, England, 1999.
\bibitem{R. A Brualdi and S. Kirkland.}  R. A Brualdi and S. Kirkland. Aztec diamonds and digraphs, and hankel determinants of schr\"{o}der numbers. J.  Combin. Theory, Series B, 94(2):334--351, 2005.
\bibitem{R.A. Brualdi and H. Schneider.} R. A. Brualdi and H. Schneider. Determinantal identities: Gauss, Schur, Cauchy, Sylvester, Kronecker, Jacobi, Binet, Laplace, Muir, and Cayley. Linear Algebra Appl., 52/53:769--791, 1983.

\bibitem{X. K. Chang-X. B. Hu and G. Xin} X. K. Chang, X. B. Hu and G. Xin. Hankel determinant solution for elliptic sequence. Linear Algebra Appl., 484 :27--45, 2015.
\bibitem{X. K. Chang. X. B. Hu and Y. N. Zhang}  X. K. Chang, X. B. Hu and Y. N. Zhang. A direct method for evaluating some nice Hankel determinants and proofs of several conjectures. Linear Algebra Appl., 438(5):2523--2541, 2013.

\bibitem{J.Cigler-Catalan}
J. Cigler. Catalan numbers, Hankel determinants and Fibonacci polynomials. Arxiv preprint arXiv:1801.05608, 2018.

\bibitem{J.Cigler}  J. Cigler. Some nice Hankel determinants. Arxiv preprint arXiv:1109.1449, 2011.

\bibitem{J.Cigler-C.Krattenthaler}  J. Cigler and C. Krattenthaler. Some determinants of path generating functions. Adv. Appl. Math., 46(1):144--174, 2011.

\bibitem {C.L. Dodgson.}  C. L. Dodgson. Condensation of determinants, being a new and brief method for computing their arithmetical values. Proceedings of the Royal Society of London, 15:150--155, 1866.
\bibitem{I. Gessel and G. Viennot.}  I. Gessel and G. Viennot. Binomial determinants, paths, and hook length formulae. Adv. in Math., 58(3):300--321, 1985.
\bibitem{Gessel and Xin}  I. M. Gessel and G. Xin. The generating function of ternary trees and continued fractions. Electronic J. Combin. 13:\# R53, 2008.


\bibitem{W. B. Jones and W. J. Thron}  W. B. Jones and W. J. Thron. Continued Fractions: Analytic Theory and Applications, Encyclopedia of Mathematics and its Applications. vol. 11, Addison-Wesley Publishing Co., Reading, Mass., 1980.



\bibitem{C. Krattenthaler.1999}  C. Krattenthaler. Advanced determinant calculus. S\'{e}minaire Lotharingien Combin., 42:B42q, 1999.

\bibitem{C. Krattenthaler.}   C. Krattenthaler. Advanced determinant calculus: a complement. Linear Algebra Appl., 411:68--166, 2005.

\bibitem{C.Krattenthaler-D.Yaqubi}   C. Krattenthaler and D. Yaqubi. Some determinants of path generating functions, II. Arxiv preprint arXiv:1802.05990.

\bibitem{J.M.E.Mays J.Wojciechowski}  M. E. Mays and J. Wojciechowski. A determinant property of Catalan numbers. Discrete Math., 211(1-3):125--134, 2000.

\bibitem{R.Sulanke and Xin} R. Sulanke and G. Xin. Hankel determinants for some common lattice paths. Adv. Appl. Math. 40:1410--167, 2008.
\bibitem{X.G.}  X. G. Viennot. Une th\'{e}orie combinatoire des polyn\^{o}mes orthogonaux g\^{e}n\'{e}rauxuqam. UQAM, Montreal, 1983.
\bibitem{H. S. Wall}  H. S. Wall. Analytic Theory of Continued Fractions. Van Nostrand, New York, 1948.

\bibitem{Wang-Xin-Catalan-power}
Y. Wang and G. Xin. Hankel determinants for convolution powers of Catalan generating functions. in preparation.

\bibitem{Xin Somos4}  G. Xin. Proof of the Somos-4 Hankel determinants conjecture. Adv. Appl. Math. 42:152--156, 2009.

\bibitem{MeiMei} M. Zhai. Circular continued fractions and Hankel determinants, M.S. Thesis. Capital Normal University, 2018. (Written in Chinese)


\end{thebibliography}
\end{document}